\documentclass{amsart}

\usepackage{graphicx,epsfig}

\usepackage{amsmath,amssymb,latexsym, amsfonts, amscd, amsthm, xy}

\usepackage{url}

\usepackage{hyperref}

\input{xy}
\xyoption{all}

\usepackage[usenames]{color}

\makeindex 

=630pt \headheight = 12pt \headsep = 20pt

\theoremstyle{plain}
\newtheorem{theorem}{Theorem}[section]
\newtheorem{conjecture}[theorem]{Conjecture}
\newtheorem{proposition}[theorem]{Proposition}
\newtheorem{corollary}[theorem]{Corollary}

\newtheorem{lemma}[theorem]{Lemma}

\theoremstyle{definition}
\newtheorem{definition}[theorem]{Definition}
\newtheorem{remark}[theorem]{Remark}
\newtheorem{example}[theorem]{Example}

\def\bF{\mathbb{F}}

\def\bQ{\mathbb{Q}}
\def\bZ{\mathbb{Z}}

\def\G{\mathbf{G}}

\def\N{\mathbf{N}}

\def\S{\mathbf{S}}

\def\cO{\mathcal{O}}

\def\cS{\mathcal{S}}

\def\fB{\mathfrak{B}}

\def\fp{\mathfrak{p}}

\def\deg{\mathrm{deg}}

\def\Norm{\mathrm{Norm}}

\begin{document}

\title[Carlitz module analogues of Mersenne primes]{Carlitz module analogues of Mersenne primes, Wieferich primes, and certain prime elements in cyclotomic function fields}

\author{Nguyen Ngoc Dong Quan}

\date{April 9, 2014}

\address{Department of Mathematics \\
         University of British Columbia \\
         Vancouver, British Columbia \\
         V6T 1Z2, Canada}

\email{\href{mailto:dongquan.ngoc.nguyen@gmail.com}{\tt dongquan.ngoc.nguyen@gmail.com}}

\maketitle

\tableofcontents

\begin{abstract}

In this paper, we introduce a Carlitz module analogue of Mersenne primes, and prove Carlitz module analogues of several classical results concerning Mersenne primes. In contrast to the classical case, we can show that there are infinitely many composite Mersenne numbers. We also study the acquaintances of Mersenne primes including Wieferich and non-Wieferich primes in the Carlitz module context that were first introduced by Dinesh Thakur.

\end{abstract}

\section{Introduction}

In the number field context, a prime $M$ is called a Mersenne prime if it is of the form $M = 2^p - 1$ for some prime $p$. The Mersenne primes are among the integers of the form $(1 + a)^m - 1$, where $a, m$ are positive integers. It is a classical result that if $(1 + a)^m - 1$ is a prime for some positive integers $a, m$ with $m \ge 2$, then it is necessary that $a = 1$ and $m = p$ for some prime $p$.

There are many strong analogies between number fields and function fields. We refer the reader to the excellent references \cite{Goss}, \cite{Rosen}, \cite{Thakur-book} for these analogies. The analogous pictures between number fields and function fields are clearly reflected when one considers the analogies between the two couples $(\bZ, \bQ)$ and $(\bF[T], \bF(T))$, where $\bF$ is a finite field. The aim of this article is to search for new analogous phenomena between number fields and function fields. Specifically we will study the notion of Mersenne primes in the Carlitz module context, and relate them to the arithmetic of cyclotomic function fields. We also study the acquaintances of Mersenne primes including Wieferich and non-Wieferich primes in the Carlitz module setting that were introduced by Thakur \cite{Thaur-Wieferich-primes} \cite{Thaur-Fermat-Wilson-congruences}.

Let us now introduce a Carlitz analogue of Mersenne primes. We begin by introducing some basic notation used here.

Let $q = p^s$, where $p$ is a prime and $s$ is a positive integer. Let $\bF_q$ be the finite field of $q$ elements. Let $A = \bF_q[T]$, and let $k = \bF_q(T)$. Let $\tau$ be the mapping defined by $\tau(x) = x^q$, and let $k\langle \tau \rangle$ denote the twisted polynomial ring. Let $C : A \rightarrow k\langle \tau \rangle$ ($a \mapsto C_a$) be the Carlitz module given by $C_T = T + \tau$. Let $R$ be a commutative $k$-algebra. The definition of the Carlitz module $C$ is equivalent to saying that $C_T(a) = Ta + a^q$ for every $a \in R$.

It is known that $C_m(x)$ is analogous to $(1 + x)^m - 1 \in \bZ[x]$. This analogy suggests the following definition: a prime in $A$ is called a \textit{Mersenne prime} if it is of the form $\alpha C_P(1)$, where $P$ is a monic prime in $A$ and $\alpha$ is a unit in $A$.

To draw an analogy between the above notion of Mersenne primes and that of Mersenne primes in the number field context, we prove in Section \ref{Section-Mersenne-primes} a Carlitz module analogue of the classical result in elementary number theory that was mentioned in the first paragraph of this introduction.

Let us now describe the content of the paper. In Section \ref{Section-Mersenne-primes}, we introduce the notions of Mersenne numbers and Mersenne primes in the Carlitz module context. As remarked in \cite{Murata-Pomerance}, it is not known whether there are infinitely many primes $p$ for which the Mersenne numbers $2^p - 1$ are composite. In contrast to the number field setting, we prove in the Carlitz module context that for every $q > 2$, there are infinitely many monic primes $\wp$ in $\bF_q[T]$ such that the Mersenne numbers $C_{\wp}(1)$ are composite.

In Section \ref{Section-Wieferich-primes}, we recall the notions of Wieferich and non-Wieferich primes in the Carlitz module context that were introduced by Dinesh Thakur \cite{Thaur-Wieferich-primes} \cite{Thaur-Fermat-Wilson-congruences}. Theorem \ref{Theorem-A-Mersenne-prime-is-a-non-Wieferich-prime-in-Carlitz-module-setting} shows that every Mersenne prime is a non-Wieferich prime, which is analogous to a similar statement in the number field context.

It is a classical result in elementary number theory that for a given odd prime $p$, every prime $q$ dividing the Mersenne number $M_p := 2^p - 1$ satisfies $q \equiv 1 \pmod{p}$. The classical proof of this result is based on the notion of the order of an element modulo a prime. In Section \ref{Section-The-Carlitz-annihilators-of-primes}, we prove a Carlitz module analogue of this result which states that for a given monic prime $P$ in $A$, every monic prime $Q$ dividing the Mersenne number $M_P := \alpha C_P(1)$ with $\alpha \in \bF_q^{\times}$ satisfies $Q \equiv 1 \pmod{P}$. In order to prove this result, we introduce in Section \ref{Section-The-Carlitz-annihilators-of-primes} a notion of the Carlitz annihilator of a prime that is analogous to that of the order of an element modulo a prime. In the last section, using the arithmetic of cyclotomic function fields, we prove a criterion for determining whether a Mersenne number is prime.

It is worth mentioning that Dinesh Thakur \cite{Thaur-Fermat-Wilson-congruences} found interesting relations linking Wieferich primes in the function field context to zeta values.

\subsection{Notation.}
\label{Subsection-Notation}

In addition to the notation introduced before, let us fix some basic notation and definitions used throughout the paper.

Every nonzero element $m \in A$ can be written in the form $m = \alpha_n T^n + \cdots + \alpha_1 T + \alpha_0$, where the $\alpha_i$ are elements in $\bF_q$ and $\alpha_n \ne 0$. When $m$ is of the form as above, we say that the degree of $m$ is $n$. In notation, we write $\deg(m) = n$. We use the standard convention that $\deg(0) = -\infty$.

For each $m \in A$, define $|m| := q^{\deg(m)}$. Note that $|m|$ is the number of elements of the finite ring $A/mA$. For basic properties of $|\cdot|$, we refer the reader to \cite{Rosen}.

Fix an algebraic closure $\bar{k}$ of $k$, and set
\begin{align*}
\Lambda := \{\lambda \in \bar{k}\; | \; \text{$C_m(\lambda) = 0$ for some nonzero $m \in A$} \}.
\end{align*}
For every nonzero element $m \in A$, define $\Lambda_m := \{\lambda \in \bar{k}\; | \; C_m(\lambda) = 0 \}$. We recall the following definition.

\begin{definition}

Let $m \in A$ be a polynomial of positive degree. The field $K_m = k(\Lambda_m)$ is called a \textit{cyclotomic function field}.

\end{definition}

For each polynomial $m$ in $A$ of positive degree, we define a \textit{primitive $m$-th root of $C_m$} to be a root of $C_m$ that generates the $A$-module $\Lambda_m$. Throughout the paper, for each $m \in A$ of positive degree, we fix a primitive $m$-th root of $C_m$, and denote it by $\lambda_m$. Let $\Phi_m$ be the $m$-th cyclotomic polynomial, that is, the monic irreducible polynomial over $k$ such that $\Phi_m(\lambda_m) = 0$.

For each polynomial $m \in A$ of positive degree, set
\begin{align*}
\cS_m := \{a \in A \; | \; \text{$\gcd(a, m) = 1$ and $0 \le \deg(a) < \deg(m)$}\}.
\end{align*}

Fix an element $m \in A$ of positive degree. For each element $a \in \cS_m$, let $\sigma_m^{(a)}$ be the $k$-automorphism of $K_m$ defined by $\sigma_m^{(a)}(\lambda_m) = C_a(\lambda_m)$. Let $\G_m$ denote the Galois group of $K_m$ over $k$. It is well-known \cite{Hayes} that $\G_m := \mathrm{Gal}(K_m/k) = \{\sigma_m^{(a)} \; | \; a \in \cS_m \}$.

\section{A Carlitz module analogue of Mersenne primes}
\label{Section-Mersenne-primes}

Let $m \in A$ be a monic polynomial of degree $d \ge 1$. By \cite[Proposition 12.11]{Rosen}, we can write $C_m(x) \in A[x]$ in the form
\begin{align}
\label{Equation-The-equation-of-Cm-x}
C_m(x) = mx + [m, 1]x^q + [m, 2]x^{q^2} + \cdots + [m, d - 1]x^{q^{d- 1}} + x^{q^d} \in A[x],
\end{align}
and hence
\begin{align*}
\dfrac{C_m(x)}{x} = m + [m, 1]x^{q - 1} + [m, 2]x^{q^2 - 1} + \cdots + [m, d - 1]x^{q^{d- 1} - 1} + x^{q^d - 1} \in A[x].
\end{align*}
It is well-known that $C_m(x) = \prod_{\substack{a | m \\ \text{$a$ monic}}}\Phi_a(x)$, where $\Phi_a(x) \in A[x]$ is the $a$-th cyclotomic polynomial.

If $m = P$ for some monic irreducible polynomial $P$, then it is well-known \cite{Hayes} that $\Phi_P(x) = C_P(x)/x$ is an Eisenstein polynomial at $P$, i.e., the polynomial $[m, i]$ is congruent to zero modulo $P$ for each $1 \le i \le d - 1$, and $m$ is divisible by $P$ but not divisible by $P^2$.

From the discussion above, we see that $C_m(x)$ is analogous to the polynomial $(1 + x)^m - 1 \in \bZ[x]$ in the classical cyclotomic theory. We now prove a lemma that naturally motivates the notions of Mersenne numbers and Mersenne primes in the Carlitz module setting.

\begin{lemma}
\label{Lemma-The-1st-lemma-that-motivates-the-definition-of-Mersenne-primes}

Let $m \in A$ be a monic polynomial of degree $d \ge 1$, and let $\alpha \in A$ be a polynomial of degree $h \ge 0$. Assume that $q \ge 3$ and $C_m(\alpha)$ is a monic prime in $A$. Then $m$ is a monic prime in $A$ and $\alpha$ belongs to $\bF_q^{\times}$.

\end{lemma}

\begin{proof}

By $(\ref{Equation-The-equation-of-Cm-x})$, we know that $C_m(\alpha) = \alpha Q$, where
\begin{align}
\label{Equation-The-equation-of-Q-in-motivation-lemma-for-Mersenne-primes}
Q = m + [m, 1]\alpha^{q - 1} + [m, 2]\alpha^{q^2 - 1} + \cdots + [m, d - 1]\alpha^{q^{d - 1} - 1} + \alpha^{q^d - 1} \in A.
\end{align}
Since $C_m(\alpha) = \alpha Q$ and $C_m(\alpha)$ is a prime, we deduce that $\alpha$ and $Q$ are nonzero.

If $\deg(\alpha) = h \ge 1$, then there exists at least one prime of positive degree in the prime factorization of $\alpha$. Since $C_m(\alpha)$ is a prime in $A$, it follows that $\alpha$ is prime, and thus $Q$ is a unit in $A$, that is, $Q$ belongs to $\bF_q^{\times}$. Therefore the degree of $Q$ is zero.

For each $1 \le i \le d - 1$, we see that the degree in $T$ of $\alpha^{q^i - 1}$ equals $h(q^i - 1)$, and thus the degree in $T$ of $[m, i]\alpha^{q^{i} - 1}$ is $q^i(d - i) + h(q^{i} - 1)$. For each $1 \le i \le d- 1$, we know from Bernoulli's inequality \cite[Theorem 42]{Hardy-Littlewood-Polya} that $h(q^{d - i} - 1) \ge q^{d - i} - 1 \ge (q - 1)(d - i) > d - i$, and thus
\begin{align*}
q^i(d - i) < hq^i(q^{d - i} - 1) = h(q^d - q^i) = h(q^d - 1) - h(q^i - 1).
\end{align*}
Hence we deduce that $h(q^d - 1) > q^i(d - i) + h(q^i - 1)$. Therefore the degree in $T$ of $\alpha^{q^d - 1}$ is greater than the degree in $T$ of $[m, i]\alpha^{q^{i} - 1}$ for each $1 \le i \le d - 1$.

Using Bernoulli's inequality and noting that $q \ge 3$, we know that $h(q^d - 1) \ge q^d - 1 \ge d(q - 1) > d$, and thus the degree in $T$ of $\alpha^{q^d - 1}$ is greater than the degree in $T$ of $m$. Therefore it follows from $(\ref{Equation-The-equation-of-Q-in-motivation-lemma-for-Mersenne-primes})$ that the degree in $T$ of $Q$ is $h(q^d - 1)$. This implies that $Q$ is not a unit in $A$, which is a contradiction. This contradiction implies that the degree of $\alpha$ is zero, and thus $\alpha$ belongs to $\bF_q^{\times}$.

We now prove that $m$ is a prime in $A$. Assume the contrary, that is, $m$ is not a prime in $A$. Hence $m$ can be written in the form $m = P_1P_2\cdots P_s$, where $s \ge 2$ and the $P_i$ are (not necessarily distinct) monic primes in $A$. For each $1 \le i \le s$, let $d_i \ge 1$ be the degree in $T$ of $P_i$. By \cite[Proposition 12.3.13]{Villa-Salvador}, we know that
\begin{align}
\label{Equation-The-factorization-of-Cm-alpha-in-cyclotomic-polynomials-Mersenne-primes-motivation}
C_m(\alpha) = \prod_{\substack{a | m \\ \text{$a$ monic}}}\Phi_a(\alpha),
\end{align}
where $\Phi_a(x) \in A[x]$ is the $a$-th cyclotomic polynomial for each monic element $a$ dividing $m$.

We know that
\begin{align}
\label{Equation-The-equation-of-Phi-P1-Mersenne-primes-motivation}
\Phi_{P_1}(\alpha) = \dfrac{C_{P_1}(\alpha)}{\alpha} = P_1 + [P_1, 1]\alpha^{q - 1} + [P_1, 2]\alpha^{q^2 - 1} + \cdots + [P_1, d_1 - 1]\alpha^{q^{d_1 - 1} - 1} + \alpha^{q^{d_1} - 1},
\end{align}
where $[P_1, i] \in A$ is a polynomial of degree $q^i(d_1 - i)$ for each $1 \le i \le d_1 - 1$.

For $1 \le i \le d_1 - 2$, we know from Bernoulli's inequality that
\begin{align*}
q^{d_1 - 1 - i} \ge 1 + (d_1 - 1 - i)(q - 1) > 1 + (d_1 - 1 - i) = d_1 - i,
\end{align*}
and it thus follows that $q^{d_1 - 1} > q^{i}(d_1 - i)$. Hence the degree in $T$ of $[P_1, d_1 - 1]$ is greater than the degree in $T$ of $[P_1, d_1 - i]$ for each $1 \le i \le d_1 - 2$. Similarly we can prove that $q^{d_1 - 1} > d_1$, and thus the degree in $T$ of $[P_1, d_1 - 1]$ is greater than the degree in $T$ of $P_1$. Since $\alpha$ is a unit in $A$, it follows from $(\ref{Equation-The-equation-of-Phi-P1-Mersenne-primes-motivation})$ that the degree in $T$ of $\Phi_{P_1}(\alpha)$ is $q^{d_1 - 1} \ge 1$. Similarly one can show that the degree in $T$ of $\Phi_{P_2}(\alpha)$ is $q^{d_2 - 1} \ge 1$. By $(\ref{Equation-The-factorization-of-Cm-alpha-in-cyclotomic-polynomials-Mersenne-primes-motivation})$, we can write
\begin{align*}
C_m(\alpha) = \Phi_{P_1}(\alpha)\Phi_{P_2}(\alpha)\prod_{\substack{a | m \\ \text{$a$ monic, $a \ne P_1, P_2$}}}\Phi_a(\alpha).
\end{align*}
Since $\Phi_{P_1}(\alpha), \Phi_{P_2}(\alpha)$ are non-units in $A$, we deduce from the last identity that $C_m(\alpha)$ is not prime in $A$, which is a contradiction. This contradiction establishes that $m$ is a prime in $A$.

\end{proof}

\begin{definition}
\label{Definition-Mersenne-numbers}

Let $q > 2$.
\begin{itemize}

\item [(i)] A Mersenne number in $A$ is a polynomial of the form $\alpha C_P(1)$, where $P$ is a monic prime in $A$ and $\alpha$ is an element in $\bF_q^{\times}$.

\item [(ii)] A Mersenne prime in $A$ is a prime of the form $\alpha C_P(1)$, where $P$ is a monic prime in $A$ and $\alpha$ is an element in $\bF_q^{\times}$.

\end{itemize}

\end{definition}

\begin{example}
\label{Example-Classification-of-Mersenne-primes-in-F-3-T}

Throughout this example, fix $q = 3$. Let $\wp_1 := T^2 + 1$, $\wp_2 := T^2 + 2T + 2$, $\wp_3 := T^2 + T + 2$. We know that $\{\wp_1, \wp_2, \wp_3\}$ consists of all monic irreducible polynomials of degree $2$ in $A$. We see that $M_{\wp_1} = C_{\wp_1}(1) = T^3 + T^2 + T + 2$, $M_{\wp_2} = C_{\wp_2}(1) = T^3 + T^2 + 2$, and $M_{\wp_3} = C_{\wp_3}(1) = T^3 + T^2 + 2T + 1$. Since the $M_{\wp_i}$ are primes in $A$, the set $\S_2 \subset A$ defined by
\begin{align*}
\S_2 := \{\alpha M_{\wp_i} \; | \; \text{$\alpha \in \bF_3^{\times}$ and $1 \le i \le 3$}\}
\end{align*}
consists of all Mersenne primes of the form $\alpha C_P(1)$, where $P$ is a monic prime in $\bF_3[T]$ of degree $2$ and $\alpha \in \bF_3^{\times}$.

\end{example}

\begin{remark}
\label{Remark-A-difficult-problem-to-determining-whether-C-P-1-is-a-prime-for-each-prime-P}

In the number field setting, it is not known whether there are infinitely many Mersenne primes. Furthermore we do not know whether there exist infinitely many prime numbers $p$ for which the Mersenne numbers $2^p - 1$ are composite (see, for example, \cite{Murata-Pomerance}). In the Carlitz module context with $q > 2$, the latter has an affirmative answer, i.e., for each $q > 2$, there exist infinitely many monic primes $\wp$ in $\bF_q[T]$ for which the Mersenne numbers $C_{\wp}(1)$ are composite. The rest of this section is devoted to proving this result.

\end{remark}

We recall a Carlitz module analogue of Fermat's little theorem which is a direct consequence of \cite[Proposition 2.4]{Hayes}.

\begin{lemma}
\label{Lemma-The-Carlitz-module-analogue-of-Fermat-little-theorem}

Let $P$ be a monic prime in $A$, and let $\alpha$ be a polynomial in $A$. Then $C_{P - 1}(\alpha) \equiv 0 \pmod{P}$.

\end{lemma}

\begin{corollary}
\label{Corollary-Fermat-little-theorem-for-alpha=1}

Let $P$ be a monic prime in $A$. Then $C_{P - 1}(1) \equiv 0 \pmod{P}$.

\end{corollary}

We recall a theorem of Hall's \cite{Hall} that plays a key role in the proof of Theorem \ref{Theorem-For-q->-2-there-are-infinitely-many-primes-P-such-that-C-P-1-is-composite}. For a proof of this result, see, for example, \cite[Theorem 4]{Pollack}.

\begin{theorem}
\label{Theorem-Theorem-of-Hall-about-twin-primes-in-function-fields}
$(\text{Hall})$

Assume that $q > 2$. Then there exists infinitely many polynomials $\wp$ in $\bF_q[T]$ such that $\wp, \wp + 1$ are monic primes.

\end{theorem}

\begin{theorem}
\label{Theorem-For-q->-2-there-are-infinitely-many-primes-P-such-that-C-P-1-is-composite}

Assume that $q > 2$. Then there exists infinitely many monic primes $\wp$ in $\bF_q[T]$ such that the Mersenne numbers $C_{\wp}(1)$ are composite.

\end{theorem}

\begin{proof}

It follows from Theorem \ref{Theorem-Theorem-of-Hall-about-twin-primes-in-function-fields} that there exist infinitely many primes $\wp, P \in A$ such that $P = \wp + 1$. Take such primes $\wp, P$ of degree $n \ge 2$ in $A$. We will prove that the Mersenne number $C_{\wp}(1)$ is composite.

By Corollary \ref{Corollary-Fermat-little-theorem-for-alpha=1}, we deduce that $C_{\wp}(1) = C_{P - 1}(1) \equiv 0 \pmod{P}$, and hence $P$ divides $C_{\wp}(1)$. Assume the contrary, that is, $C_{\wp}(1)$ is a prime. Since $P$ divides  $C_{\wp}(1)$, it follows that $C_{\wp}(1) = \alpha P$ for some unit $\alpha$ in $A$. Hence
\begin{align}
\label{Equation-The-degree-of-C-wp-at-1-is-n}
n = \deg(P) = \deg(C_{\wp}(1)).
\end{align}
By \cite[Proposition 12.11]{Rosen} and since the degree of $\wp$ is $n$, we can write $C_{\wp}(x) \in A[x]$ in the form
\begin{align*}
C_{\wp}(x) = \wp x + [\wp, 1]x^q + [\wp, 2]x^{q^2} + \cdots + [\wp, n - 1]x^{q^{n - 1}} + x^{q^{n}},
\end{align*}
where $[\wp, i]$ is a polynomial of degree $q^i(n - i)$ for each $1 \le i \le n - 1$. Hence
\begin{align}
\label{Equation-C-wp-at-1}
C_{\wp}(1) = \wp + [\wp, 1] + [\wp, 2] + \cdots + [\wp, n - 1] + 1.
\end{align}
Using Bernoulli's inequality, we see that $q^{n - i - 1} \ge 1 + (q - 1)(n - i - 1)$ for every $0 \le i \le n - 2$. Since $q > 2$ and $n - i - 1 \ge 1$ for each $0 \le i \le n - 2$, we deduce that $1 + (q - 1)(n - i - 1) > n - i$, and thus $q^{n - i - 1} >  n - i$. Therefore
\begin{align*}
\deg([\wp, n - 1]) = q^{n - 1} > q^i(n - i) = \deg([\wp, i])
\end{align*}
for every $0 \le i \le n - 2$. It thus follows from $(\ref{Equation-C-wp-at-1})$ that the degree of $C_{\wp}(1)$ is $q^{n - 1}$, which contradicts $(\ref{Equation-The-degree-of-C-wp-at-1-is-n})$. Thus $C_{\wp}(1)$ is composite, which proves our contention.

\end{proof}

\section{Wieferich primes and non-Wieferich primes}
\label{Section-Wieferich-primes}

In this section, we recall the notion of Wieferich primes in $A$ that was introduced by Thakur \cite{Thaur-Wieferich-primes} \cite{Thaur-Fermat-Wilson-congruences}. The aim of this section is to prove that a Mersenne prime is a non-Wieferich prime, which is analogous to a similar statement in the number field context. We begin by recalling the notions of Wieferich primes and non-Wieferich primes in the function field context.

\begin{definition}
\label{Definition-Wieferich-primes}
(see \cite{Thaur-Wieferich-primes} and \cite{Thaur-Fermat-Wilson-congruences})

Let $q > 2$, and let $\wp$ be a prime in $A$. Let $\alpha$ be the leading coefficient of $\wp$, and let $P$ be the unique monic prime in $A$ such that $\wp = \alpha P$. The prime $\wp$ is called \textit{a Wieferich prime} if $C_{P - 1}(1) \equiv 0 \pmod{P^2}$.

\end{definition}

\begin{definition}
\label{Definition-Non-Wieferich-primes}
(see \cite{Thaur-Wieferich-primes} and \cite{Thaur-Fermat-Wilson-congruences})

We maintain the same notation and assumptions as in Definition \ref{Definition-Wieferich-primes}. The prime $\wp$ is called \textit{a non-Wieferich prime} if $C_{P - 1}(1) \not\equiv 0 \pmod{P^2}$.

\end{definition}

In the number field case, it is a classical result that a Mersenne prime $p \in \bZ$ is a non-Wieferich prime. We now prove a Carlitz module analogue of this result.

\begin{theorem}
\label{Theorem-A-Mersenne-prime-is-a-non-Wieferich-prime-in-Carlitz-module-setting}

Assume that $q > 2$. Let $M_P = \alpha C_P(1)$ be a Mersenne prime, where $\alpha$ is a unit in $A$ and $P$ is a monic prime in $A$. Then $M_P$ is a non-Wieferich prime.

\end{theorem}

\begin{proof}

Let $\beta \in \bF_q^{\times}$ be the leading coefficient of $M_P$, and let $\wp$ be the unique monic prime in $A$ such that $M_P = \beta \wp$. Define
\begin{align}
\label{Equation-The-definition-of-gamma-in-lemma-that-Mersenne-primes-are-non-Wieferich-primes}
\gamma = \alpha^{-1}\beta \in \bF_q^{\times}.
\end{align}
We see that
\begin{align}
\label{Equation-C-P-1-equals-alpha-inverse-times-beta-times-wp}
C_P(1) = \alpha^{-1}M_P = \gamma \wp.
\end{align}
By Corollary \ref{Corollary-Fermat-little-theorem-for-alpha=1}, we know that $\gamma \wp - 1 = C_P(1) - 1 = C_{P - 1}(1) \equiv 0 \pmod{P}$, and thus there exists a nonzero element $Q$ in $A$ such that
\begin{align}
\label{Equation-alpha-inverse-times-beta-wp-minus-1-equals-PQ}
\gamma \wp  - 1 = PQ.
\end{align}

Assume the contrary, that is, $M_P$ is a Wieferich prime. Therefore we deduce that
\begin{align*}
C_{PQ}(1) = C_{\gamma \wp  - 1}(1) = C_{\gamma\wp}(1) - 1 = \gamma C_{\wp}(1) - 1 \equiv  \gamma - 1 \pmod{\wp^2}.
\end{align*}
Hence it follows from $(\ref{Equation-C-P-1-equals-alpha-inverse-times-beta-times-wp})$ that $C_Q(\gamma \wp) = C_Q(C_P(1)) = C_{PQ}(1) \equiv \gamma - 1 \pmod{\wp^2}$, and therefore
\begin{align}
\label{Equation-The-1st-equation-in-lemma-that-Mersenne-primes-are-non-Wieferich-primes}
C_Q(\gamma \wp) - (\gamma - 1) \equiv 0 \pmod{\wp^2}.
\end{align}

Let $r \ge 0$ be the degree of $Q$. If $r = 0$, we see that $Q$ is a unit in $\bF_q^{\times}$. Hence $C_Q(\gamma \wp) = Q\gamma \wp$, and it thus follows from $(\ref{Equation-The-1st-equation-in-lemma-that-Mersenne-primes-are-non-Wieferich-primes})$ that $Q\gamma \wp - (\gamma - 1) \equiv 0 \pmod{\wp^2}$. Thus $\wp^2$ divides $Q\gamma \wp - (\gamma - 1)$, which implies that
\begin{align*}
\deg(\wp) = \deg(Q\gamma \wp - (\gamma - 1)) \ge \deg(\wp^2) = 2\deg(\wp).
\end{align*}
Hence $\deg(\wp) = 0$, which is a contradiction since $\wp$ is a monic prime of positive degree. Thus $\deg(Q) = r \ge 1$.

We can write $C_Q(x) \in A[x]$ in the form $C_Q(x) = Qx + [Q, 1]x^q + \cdots + [Q, r - 1]x^{q^{r-1}} + [Q, r]x^{q^r}$, where $[Q, i] \in A$ is a polynomial of degree $q^i(r - i)$ for each $1 \le i \le r - 1$ and $[Q, r] \in \bF_q^{\times}$ is the leading coefficient of $Q$. From the equation of $C_Q(x)$, we deduce that
\begin{align}
\label{Equation-The-2nd-equation-in-lemma-that-Mersenne-primes-are-non-Wieferich-primes}
\dfrac{C_Q(\gamma \wp) - (\gamma - 1)}{\wp} = \dfrac{\gamma C_Q(\wp) - (\gamma - 1)}{\wp} = \gamma(Q + [Q, 1]\wp^{q - 1} + \cdots + [Q, r]\wp^{q^r - 1}) - \dfrac{(\gamma - 1)}{\wp}.
\end{align}

By $(\ref{Equation-The-1st-equation-in-lemma-that-Mersenne-primes-are-non-Wieferich-primes})$, we deduce that
\begin{align}
\label{Equation-The-3rd-equation-in-lemma-that-Mersenne-primes-are-non-Wieferich-primes}
\dfrac{C_Q(\gamma \wp) - (\gamma - 1)}{\wp} \equiv 0 \pmod{\wp},
\end{align}
and thus there is an element $R$ in $A$ such that $\dfrac{C_Q(\gamma \wp) - (\gamma - 1)}{\wp} = \wp R \in A$. We deduce from $(\ref{Equation-The-2nd-equation-in-lemma-that-Mersenne-primes-are-non-Wieferich-primes})$ that $\dfrac{(\gamma - 1)}{\wp} \in A$, and thus $\wp$ divides $\gamma - 1$. Since $\gamma - 1$ belongs to $\bF_q$ and $\wp$ is a monic prime of positive degree, we deduce that $\gamma - 1 = 0$, and hence $\gamma = 1$.

For $1 \le i \le r$, we see that $q^i - 1 \ge 1$, and it thus follows that $\wp^{q^i - 1} \equiv 0 \pmod{\wp}$. By $(\ref{Equation-The-2nd-equation-in-lemma-that-Mersenne-primes-are-non-Wieferich-primes})$, $(\ref{Equation-The-3rd-equation-in-lemma-that-Mersenne-primes-are-non-Wieferich-primes})$ and since $\gamma = 1$, we deduce from the last congruences that $0 \equiv \dfrac{C_Q(\wp)}{\wp} \equiv Q \pmod{\wp}$. Therefore there exists a nonzero element $P_1 \in A$ such that $P_1\wp = Q$.

By $(\ref{Equation-alpha-inverse-times-beta-wp-minus-1-equals-PQ})$ and since $\gamma = 1$, we deduce that $\wp - 1 = PQ = P(P_1\wp)$, and thus $\wp(1 - P_1P) = 1$. Therefore $\wp$ is a unit in $A$, which is a contradiction. Thus $M_P$ is a non-Wieferich prime.

\end{proof}

Thakur \cite{Thaur-Fermat-Wilson-congruences} recently found interesting relations linking Wieferich primes in the function field context to zeta values. Also in \cite{Thaur-Fermat-Wilson-congruences}, Thakur made the following conjecture that relates the degree of a Wieferich prime to the characteristic $p$.
\begin{conjecture}
$(\text{Thakur's conjecture})$

For every $p > 2$, the degree of a Wieferich prime in $A = \bF_q[T]$ is divisible by the characteristic $p$.

\end{conjecture}

\section{The Carlitz annihilators of primes}
\label{Section-The-Carlitz-annihilators-of-primes}

In this section, we introduce a notion of the Carlitz annihilator of a prime. By way of illustration, we prove a Carlitz module analogue of a classical result concerning congruences for prime divisors of the Mersenne numbers in elementary number theory.

For a monic prime $P \in A$, we know from Corollary \ref{Corollary-Fermat-little-theorem-for-alpha=1} that $C_{P - 1}(1) \equiv 0 \pmod{P}$. Hence using Proposition 1.6.5 and Lemma 1.6.8 in \cite{Goss}, the following result follows immediately.

\begin{proposition}
\label{Proposition-The-result-motivating-the-definition-of-Carlitz-order-of-a-prime}

Let $P$ be a monic prime in $A$ of positive degree. Then there exists a unique monic polynomial $\wp_{P} \in A$ of positive degree satisfying the following.
\begin{itemize}

\item [(i)] $C_{\wp_{P}}(1) \equiv 0 \pmod{P}$;

\item [(ii)] for any nonzero element $a \in A$, $\wp_{P}$ divides $a$ if and only if $C_a(1) \equiv 0 \pmod{P}$; and

\end{itemize}

\end{proposition}

Proposition \ref{Proposition-The-result-motivating-the-definition-of-Carlitz-order-of-a-prime} motivates the following definition that is a Carlitz module analogue of the order of an element in a finite group.

\begin{definition}
\label{Definition-The-Carlitz-module-order-of-a-prime-in-F-q}

Let $P$ be a monic prime in $A$ of positive degree, and let $\wp_P$ be the unique monic polynomial satisfying $(i), (ii)$ in Proposition \ref{Proposition-The-result-motivating-the-definition-of-Carlitz-order-of-a-prime}. The polynomial $\wp_P$ is called \textit{the Carlitz annihilator of $P$}.

\end{definition}

By Corollary \ref{Corollary-Fermat-little-theorem-for-alpha=1} and part $(ii)$ of Proposition \ref{Proposition-The-result-motivating-the-definition-of-Carlitz-order-of-a-prime}, the following result follows immediately.

\begin{lemma}
\label{Lemma-The-Carlitz-order-of-P-divides-P-minus-1}

Let $P$ be a monic prime in $A$, and let $\wp_P$ be the Carlitz annihilator of $P$. Then $\wp_P$ divides $P - 1$.

\end{lemma}

Return to the number field context, and let $M_p := 2^p - 1$ be a Mersenne number for some odd prime $p$. A classical result in elementary number theory says that any prime $q$ dividing $M_p$ satisfies $q \equiv 1 \pmod{p}$ (for a proof of this result, see, for example, \cite[Theorem 3]{Jaroma-Reddy}). We now prove a Carlitz module analogue of this result using the notion of the Carlitz annihilator of a prime in a similar manner that the notion of the order of an integer modulo a prime appears in the proof of the analogous result in the number field context mentioned above.

\begin{theorem}
\label{Theorem-Every-prime-dividing-a-Carlitz-module-Mersenne-number-is-congruent-to-1-mod-P}

Let $P$ be a monic prime in $A$, and let $M_P := \alpha C_P(1)$ be a Mersenne number, where $\alpha$ is an element in $\bF_q^{\times}$. Then $Q \equiv 1 \pmod{P}$ for any monic prime $Q$ dividing $M_P$.

\end{theorem}

\begin{proof}

Let $Q$ be any monic prime in $A$ such that $Q$ divides $M_P$. Thus $M_P = \alpha C_P(1) \equiv 0 \pmod{Q}$. Since $\alpha$ is a unit in $A$ and $Q$ is a monic prime, it follows that
\begin{align}
\label{Equation-The-1st-equation-in-the-result-about-being-congruent-to-1-mod-P-of-primes-dividing-Mersenne-numbers}
C_P(1) \equiv 0 \pmod{Q}.
\end{align}

Let $\wp_{Q}$ be the Carlitz annihilator of $Q$. By part $(ii)$ of Proposition \ref{Proposition-The-result-motivating-the-definition-of-Carlitz-order-of-a-prime} and $(\ref{Equation-The-1st-equation-in-the-result-about-being-congruent-to-1-mod-P-of-primes-dividing-Mersenne-numbers})$, we deduce that $\wp_Q$ divides $P$. Since $\wp_Q$ is a monic polynomial of positive degree and $P$ is a monic prime in $A$, it follows that $\wp_Q = P$. By Lemma \ref{Lemma-The-Carlitz-order-of-P-divides-P-minus-1}, we know that $\wp_Q$ divides $Q - 1$, and thus $P$ divides $Q - 1$. Therefore $Q \equiv 1 \pmod{P}$, which proves our contention.

\end{proof}

As a by-product of Theorem \ref{Theorem-Every-prime-dividing-a-Carlitz-module-Mersenne-number-is-congruent-to-1-mod-P}, one obtains the following result that precisely describes the Carlitz annihilator of any monic prime occurring in the prime factorization of a Mersenne number.

\begin{corollary}
\label{Corollary-The-Carlitz-orders-of-primes-in-the-prime-factorization-of-a-Mersenne-number}

Let $P$ be a monic prime in $A$, and let $M_P := \alpha C_P(1)$ be a Mersenne number, where $\alpha$ is an element in $\bF_q^{\times}$. Let $Q$ be a monic prime in $A$ such that $Q$ divides $M_P$, and let $\wp_Q$ be the Carlitz annilator of $Q$. Then $\wp_Q = P$.

\end{corollary}

\begin{proof}

Following the proof of Theorem \ref{Theorem-Every-prime-dividing-a-Carlitz-module-Mersenne-number-is-congruent-to-1-mod-P}, we derive Corollary \ref{Corollary-The-Carlitz-orders-of-primes-in-the-prime-factorization-of-a-Mersenne-number} immediately.

\end{proof}

The following result fully characterizes all monic primes $P$ for which $\wp_P$ is a prime.

\begin{corollary}
\label{Corollary-A-criterion-for-determining-when-wp-P-is-a-prime}

Let $P$ be a monic prime in $A$, and let $\wp_P$ be the Carlitz annihilator of $P$. Then $\wp_P$ is a prime if and only if $P$ divides $C_Q(1)$ for some monic prime $Q$.

\end{corollary}

\begin{proof}

Assume that $P$ divides $C_Q(1)$ for some monic prime $Q$. It follows from Corollary \ref{Corollary-The-Carlitz-orders-of-primes-in-the-prime-factorization-of-a-Mersenne-number} that $\wp_P = Q$, and thus $\wp_P$ is a prime.

Conversely assume that $\wp_P$ is a prime, say $Q$. By Proposition \ref{Proposition-The-result-motivating-the-definition-of-Carlitz-order-of-a-prime}, we see that $C_Q(1) = C_{\wp_P}(1) \equiv 0 \pmod{P}$, and thus $P$ divides $C_Q(1)$.
\end{proof}

\section{A criterion for determining whether a Mersenne number is prime}
\label{Section-A-criterion-for-determing-whether-a-Mersenne-number-is-prime}

In this section, we prove a criterion for determining whether a Mersenne number is prime. The criterion relies on primality of certain elements in cyclotomic function fields. For the notation used in this section, we refer the reader to Subsection \ref{Subsection-Notation}.

\begin{theorem}
\label{Theorem-A-criterion-for-determining-whether-a-Mersenne-number-is-prime}

Let $P$ be a monic prime in $A$. Let $M_P := \alpha C_P(1)$ be a Mersenne number, where $\alpha$ is an element in $\bF_q^{\times}$. Let $K_P$ be the $P$-th cyclotomic function field. Then $M_P$ is a prime in $A$ if and only if $1 - \lambda_P$ is a prime element in the ring of integers $\cO_P$ of $K_P$. Furthermore when $M_P$ is a prime, $M_P$ splits completely in $K_P$.

\end{theorem}

\begin{proof}

Throughout the proof, let $\Phi_P$ be the $P$-th cyclotomic polynomial. It is known \cite[Proposition 2.4]{Hayes} that $M_P = \alpha C_P(1) = \alpha \Phi_P(1)$.

Assume that $M_P$ is a prime in $A$. It follows that $\Phi_P(1)$ is a prime in $A$. We know that
\begin{align}
\label{Equation-Representation-of-Phi-P-1-as-the-norm-of-1-minus-lambda-P}
\Phi_P(1) = \prod_{\substack{0 \le \deg(a) < \deg(P)}}(1 - \sigma_P^{(a)}(\lambda_P)) = \Norm_{K_P/k}(1 - \lambda_P).
\end{align}

Let
\begin{align}
\label{Equation-The-prime-ideal-factorization-of-1-minus-lambda-P}
(1 - \lambda_P)\cO_P = \fB_1^{\epsilon_1}\fB_2^{\epsilon_2}\cdots\fB_h^{\epsilon_h}
\end{align}
be the prime ideal factorization of the principal ideal $(1 - \lambda_P)\cO_P$ in $\cO_P$, where the $\fB_i$ are distinct prime ideals in $\cO_P$ and the $\epsilon_i$ are positive integers. For each $1 \le i \le h$, let $\fp_i$ be the prime ideal of $k$ lying below $\fB_i$, and let $f_i := f(\fB_i/\fp_i)$ denote the relative degree of $\fB_i$ over $\fp_i$. Let $\N_{K_P/k}$ denote the norm map on ideals of $K_P$ over $k$.

We know from $(\ref{Equation-Representation-of-Phi-P-1-as-the-norm-of-1-minus-lambda-P})$, $(\ref{Equation-The-prime-ideal-factorization-of-1-minus-lambda-P})$ and \cite[Theorem 3.1.3]{Goldschmidt} that
\begin{align*}
\Phi_P(1)A = \Norm_{K_P/k}(1 - \lambda_P)A = \N_{K_P/k}((1 - \lambda_P)\cO_P) = \N_{K_P/k}(\prod_{i = 1}^h\fB_i^{\epsilon_i}) = \prod_{i = 1}^h\fp_i^{\epsilon_i f_i}.
\end{align*}
Since $\Phi_P(1)$ is a prime in $A$, we derive that $\Phi_P(1)A$ is a prime ideal of $A$. Thus it follows from the above equation that $h = \epsilon_1 = f_1 = 1$. Therefore we deduce from $(\ref{Equation-The-prime-ideal-factorization-of-1-minus-lambda-P})$ that $(1 - \lambda_P)\cO_P = \fB_1$, which is a prime ideal. Hence $1 - \lambda_P$ is a prime element in $\cO_P$.

Conversely suppose that $1 - \lambda_P$ is a prime element in $\cO_P$. Hence there is a prime ideal $\fB$ in $\cO_P$ such that $(1 - \lambda_P)\cO_P = \fB$. Let $\fp$ be the prime ideal of $k$ lying below $\fB$, and let $f := f(\fB/\fp)$ be the relative degree of $\fB$ over $\fp$. Since $A$ is a principal ideal domain, there exists a monic prime $\wp$ in $A$ such that $\fp = \wp A$. Using \cite[Theorem 3.1.3]{Goldschmidt} and repeating the same arguments as above, we deduce that
\begin{align*}
\Phi_P(1)A = \Norm_{K_P/k}(1 - \lambda_P)A = \N_{K_P/k}((1 - \lambda_P)\cO_P) = \N_{K_P/k}(\fB) = \fp^{f} = (\wp A)^f = \wp^f A.
\end{align*}
Thus $\wp^f$ divides $\Phi_P(1)$, and hence $\wp$ divides $\Phi_P(1)$. Hence $\wp$ divides $M_P$, and it thus follows from Theorem \ref{Theorem-Every-prime-dividing-a-Carlitz-module-Mersenne-number-is-congruent-to-1-mod-P} that $\wp \equiv 1 \pmod{P}$. By \cite[Theorem 12.10]{Rosen}, we deduce that $\wp A$ splits completely in $K_P$, and therefore $f = 1$. Hence $\Phi_P(1)A = \wp A$, which implies that $\Phi_P(1)$ is a prime in $A$. Therefore $M_P$ is a prime.

Finally when $M_P$ is a prime, we see that since $\alpha$ is a unit in $A$, there exists a monic prime $\wp$ in $A$ such that $M_PA = \alpha \Phi_P(1)A = \wp A$. Using the same arguments as above, we deduce that $\wp \equiv 1 \pmod{P}$, and therefore $M_P A = \wp A$ splits completely in $K_P$.

\end{proof}

\begin{remark}

Theorem \ref{Theorem-A-criterion-for-determining-whether-a-Mersenne-number-is-prime} is a Carlitz module analogue of the corollary to \cite[Theorem 3]{Helou} in the classical cyclotomic theory. Although we have adapted the ideas in the proof of the corollary in \cite{Helou} to the proof of Theorem \ref{Theorem-A-criterion-for-determining-whether-a-Mersenne-number-is-prime} in the Carlitz module context, there is a substantial difference between the two proofs. To be more specific, note that the congruence $\wp \equiv 1 \pmod{P}$ is a key step in both the proof of Theorem \ref{Theorem-A-criterion-for-determining-whether-a-Mersenne-number-is-prime} and that of the corollary to \cite[Theorem 3]{Helou}, where $\wp$ is any prime dividing the Mersenne number $M_P$. In the number field case, this congruence is a well-known result and can be derived using the notion of the orders of elements modulo primes (see \cite[Theorem 3]{Jaroma-Reddy} for a proof of this congruence). The congruence $\wp \equiv 1 \pmod{P}$ in the number field context can also be proved using some facts about primes dividing the norms of certain cyclotomic elements as shown in \cite{Helou}. In the Carlitz module context, we need to introduce the notion of the Carlitz annihilator of a prime as a replacement for that of the order of an element modulo a prime to derive the congruence $\wp \equiv 1 \pmod{P}$.

\end{remark}

\section*{Acknowledgements}

I am very grateful to Dinesh Thakur for many of his insights in function field arithmetic, making many useful comments, and pointing out some useful references. I thank the referee for useful comments. I thank my parents, Nguyen Ngoc Quang and Phan Thi Thien Huong, for their constant support. I was supported by a postdoctoral fellowship in the Department of Mathematics at University of British Columbia.


\begin{thebibliography}{179}


\bibitem{Goldschmidt} {\sc D.M. Goldschmidt}, \emph{Algebraic functions and projective curves}, Graduate Texts in Mathematics, {\bf215}. Springer-Verlag, New York (2003).

\bibitem{Goss} {\sc D. Goss}, \emph{Basic structures of function field arithmetic}, Ergebnisse der Mathematik und ihrer Grenzgebiete ({\bf3}) [Results in Mathematics and Related Areas ({\bf3})], {\bf35}, Springer-Verlag, Berlin, (1996).


\bibitem{Hall} {\sc C. Hall}, \emph{$L$-functions of twisted Legendre curves}, J. Number Theory {\bf119} (2006), no. {\bf1}, 128--147.




\bibitem{Hardy-Littlewood-Polya} {\sc G.H. Hardy, J.E. Littlewood and G. P\'olya}, \emph{Inequalities}, Cambridge University Press, Cambridge, UK, (1952).


\bibitem{Hayes} {\sc D.R. Hayes}, \emph{Explicit class field theory for rational function fields}, Trans. Amer. Math. Soc. {\bf189} (1974), pp. 77--91.


\bibitem{Helou} {\sc C. Helou}, \emph{Reciprocal relations between cyclotomic fields}, J. Number Theory {\bf130} (2010), no.{\bf8}, 1854--1875.






\bibitem{Jaroma-Reddy} {\sc J.H. Jaroma and K.N. Reddy}, \emph{Classical and alternative approaches to the Mersenne and Fermat numbers}, Amer. Math. Monthly {\bf114} (2007), no.{\bf8}, 677--687.

\bibitem{Mauduit} {\sc V. Mauduit}, \emph{Carmichael-Carlitz polynomials and Fermat-Carlitz quotients}, Finite fields and applications (Glasgow, 1995), 229--242, London Math. Soc. Lecture Note Ser., {\bf233}, Cambridge Univ. Press, Cambridge, (1996).


\bibitem{Murata-Pomerance} {\sc L. Murata and C. Pomerance}, \emph{On the largest prime factor of a Mersenne number}, Number theory, 209--218,
CRM Proc. Lecture Notes, 36, Amer. Math. Soc., Providence, RI, (2004).



\bibitem{Pollack} {\sc P. Pollack}, \emph{Simultaneous prime specializations of polynomials over finite fields}, Proc. Lond. Math. Soc. ({\bf3}) {\bf97} (2008), no.{\bf3}, 545--567.



\bibitem{Rosen} {\sc M. Rosen}, \emph{Number theory in function fields}, Graduate Texts in Mathematics, {\bf210}. Springer-Verlag, New York (2002).


\bibitem{SSTT} {\sc J. Sauerberg, L. Shu, D.S. Thakur, and G. Todd}, \emph{Infinitude of Wilson primes for $\bF_q[t]$}, Acta Arith. {\bf157} (2013), no.{\bf1}, 91--100.



\bibitem{Thaur-Wieferich-primes} {\sc D.S. Thakur}, \emph{Iwasawa theory and cyclotomic function fields}, In Arithmetic Geometry (Tempe, AZ 1993), vol. {\bf174} of Contemp. Math. 157--165, Amer. Math. Soc. (1994).



\bibitem{Thakur-book} {\sc D.S. Thakur}, \emph{Function Field Arithmetic}, World Scientific Publishing Co., Inc., River Edge, NJ, (2004).



\bibitem{Thakur1} {\sc D.S. Thakur}, \emph{Differential characterization of Wilson primes for $\bF_q[t]$}, Algebra \& Number Theory {\bf7} (2013), no. {\bf8}, 1841--1848.


\bibitem{Thaur-Fermat-Wilson-congruences} {\sc D.S. Thakur}, \emph{Fermat-Wilson congruences and zeta values}, preprint (2013).


\bibitem{Villa-Salvador} {\sc G.D. Villa Salvador}, \emph{Topics in the theory of algebraic function fields}, Mathematics: Theory \& Applications. Birkh\"auser Boston, Inc., Boston, MA (2006).



\end{thebibliography}
\end{document}